\documentclass[12pt]{amsart}
\usepackage{amsmath,amssymb,amscd}
\usepackage{epsfig}
\usepackage{epsfig, graphicx,pinlabel}

\newtheorem{theorem}{Theorem}

\newtheorem{prop}{Proposition}

\newtheorem{lemma}{Lemma}

\newtheorem*{theoremquote}{Theorem}

\theoremstyle{definition}

\newtheorem{remark}{Remark}

\newtheorem{definition}{Definition}

\newcommand\setthm[1]{\renewcommand\thetheorem{#1}}

\def\a{\alpha}
\def\b{\beta}

\def\bR{\mathbb R}

\def\cB{\mathcal B}
\def\cJ{\mathcal J}
\def\cA{\mathcal A}

\def\rank{\textup{rank}\, }

\input xy
\xyoption{all}

\input epsf
\def\girc{\gamma^{(i)}_{r,c}}
\def\cgirc{\overline{\gamma}^{(i)}_{r,c}}

\begin{document}

\title[Characteristic polynomial of $\mathcal J_n$]
{Enumeration of graphs and the characteristic polynomial of the hyperplane arrangements $\mathcal J_n$}


\author{Joungmin Song}
\curraddr{Division of Liberal Arts \& Sciences, GIST\\ Gwangju, 61005, Korea}\address{}

\email{songj@gist.ac.kr}
\thanks{}



\date{}

\dedicatory{}

\commby{}

\begin{abstract}
We give a complete formula for the characteristic polynomial of  hyperplane arrangements $\cJ_n$ consisting of the hyperplanes $x_i+x_j=1$, $x_k=0$, $x_l=1$, $ 1\leq i, j, k, l\leq n$.
The formula is obtained by associating hyperplane arrangements with graphs, and then enumerating central graphs via generating functions for the  number of bipartite graphs of given order, size and number of connected components.

\end{abstract}

\maketitle

\section{Introduction and preliminaries}

In this paper, we give a complete formula for the characteristic polynomial of the hyperplane arrangement $\mathcal J_n$ consisting of
\begin{enumerate}
\item  the walls or hyperplanes of {\bf type I}: $ H_{\a\b} = \{x\in \bR^n| x_\a +x_\b = 1\} = H_{\b\a}$, $1\leq \a, \b\leq n$;
\item  the  walls of {\bf type II}: $  0_i:=\{x\in \bR^n |x_i=0\}, \mbox{  and  }\ 1_i :=\{x\in \bR^n|x_i=1\}$, $\forall i \in [n]:= \{1, 2, \dots, n\}$.
    \end{enumerate}
  This particular hyperplane arrangement was first considered in \cite{SONG1} where the key idea of associating the sub-arrangements of $\mathcal J_n$ with certain colored graphs was developed,   and a few basic examples were worked out. We further advanced the method in the subsequent papers \cite{SONG2, SONG3}, and with the complete formula for the characteristic polynomial, we accomplish the eventual goal of the project which is  to give an explicit formula for the number of chambers i.e. the connected components of the complement of the hyperplanes.

There are a few powerful methods for computing the characteristic polynomial of hyperplane arrangements. In \cite{Atha}, Athanasiadis showed that simply by computing the number of points missing the hyperplanes over finite fields and by using the M\"obius formula, the characteristic polynomials of many hyperplane arrangements can be computed. Our subject of interest $\mathcal J_n$ is not among the classes of examples considered in \cite{Atha} and it is indeed not even a deformation of them, but we believe that the finite field method should be applicable in our case too. Our method is closer in spirit to \cite{Zas-signed}, who showed that the characteristic polynomial is equal to the chromatic polynomial of some associated signed graphs.

For the sake of completeness, we shall briefly recall the basic definitions and main results from our previous papers \cite{SONG1,SONG2, SONG3}.
  A hyperplane arrangement $\cB$  is said to be {\it central} if the intersection of all hyperplanes in $\cB$ is nonempty.   The {\it rank} of a hyperplane arrangement is the dimension of the space spanned by the normal vectors to the hyperplanes in the arrangement.   The {\it characteristic polynomial} of an arrangement $\cA$ is defined
\[ \chi_\cA(t) = \sum_{\cB} (-1)^{|\cB|} t^{n-\rank(\cB)}\] where $\cB$ runs through all central subarrangements of $\cA$.  The importance of this polynomial is revealed in one of the most fundamental theorems in the theory of hyperplane arrangements:
\begin{theoremquote}\label{T:Zas75}\cite{Zas75}

Let $\mathcal A$ be a hyperplane arrangement in an $n$-dimensional real vector space. Let $r(\mathcal A)$ be the number of chambers and $b(\mathcal A)$  be the number of relatively bounded chambers. Then we have
\begin{enumerate}
\item $b(\mathcal A) = (-1)^{n}\chi(+1)$.
\item $r(\mathcal A)= (-1)^{n}\chi(-1)$.
\end{enumerate}
\end{theoremquote}

In \cite{SONG1}, we considered 3-colored graphs on the vertex set $[n]=\{1,2,\dots, n\}$, and defined the {\it centrality} of graphs by specifying the parity of the paths between two given colored vertices. To a hyperplane arrangement $\mathcal A$, we corresponded the graph $\Gamma_{\mathcal A}$ whose vertex set is
\[
\displaystyle{
I(\mathcal A) = \left(\bigcup_{H_{ij} \in \mathcal A} \{i,j\}\right) \cup \left(\bigcup_{0_\a\in \mathcal A} \{\a\}\right) \cup \left( \bigcup_{1_\b \in \mathcal A} \{\b\}\right)},
\]
the edge set is  $\{\{i,j\} \ | \ H_{ij} \in \mathcal A\}$, and its vertex $i$ is colored $0$ or $1$ or $ *$ respectively if $0_i \in \mathcal A$ or $1_i \in \mathcal A$, or $0_i, 1_i \not\in \mathcal A$ respectively. The main theorem that begun it all is:
\begin{theoremquote} \cite{SONG1} $\cA$ is central if and only if $\Gamma_\cA$ is central.
\end{theoremquote}
This further lead our investigation to procure methods of enumerating central graphs. In the main theorem of \cite{SONG2}, we gave a formula for the coefficients of $\chi_{\mathcal J_n}$ in terms of the number of connected graphs  and the number of connected bipartite graphs of given order and size.   In \cite{SONG3}, we gave a full description of the number of bipartite graphs using the Exponential Formula \cite[Corollary~5.1.6]{StanleyBook2}. Enumeration of bipartite graphs have been studied by many authors  (\cite{Harary58, Harary63, Harary73, Hanlon, StanleyBook2, Ueno} to name just several),  but our search did not turn up a comprehensive list of the number of bipartite graphs of give order, size and number of connected components. By combining the main results of \cite{SONG2} and \cite{SONG3}, we are now able to give a complete generating function for the coefficients of $\chi_{\mathcal J_n}$.

Our two main results are as follows. Let $\bar\gamma^{(0)}_{r,c}$ denote the number of connected, non-colored, bipartite graphs without isolated vertices whose rank and cardinality are $r$ and $c$. Let
 $\bar b_{n,k}$ be the number of connected bipartite graphs of order $n$ and size $k.$

\begin{theorem}\label{T:main}
 The generating function for the number of central graphs is given by
\[
\begin{array}{l}
\Gamma(x,y,z)  =  \exp\left[ \left( \frac12 \log\left(1 + \sum_{n\ge 1, k\ge 0} \sum_{i=0}^n \binom ni \binom{i(n-i)}k \frac 1{n!}x^ny^k\right) - x\right) \frac zx \right] \cdot \\
\left( \sum_{r=0}^\infty \frac{2^r}{r!} x^r y^r\right) \cdot
\left( \sum \left( \sum_{t=1}^{c-r} 2 \overline{\gamma}^{(0)}_{r-1,c-t}\binom rt \right) \frac1{r!} x^r y^c\right) \cdot \\
\left(\exp\left(\log \left(1+ \sum_{n \ge 1, k \ge 0} \binom{\binom n2}k \frac1{n!} x^n y^k\right) - x \right) - \sum_{n\ge 2, k\ge 1} \bar b_{n,k} \frac1{n!}x^ny^k\right)
\end{array}
\]
\end{theorem}

\

\begin{theorem}\label{T:char-poly}
The characteristic polynomial of $\mathcal J_n$ is given by
\[
\chi_{\mathcal J_n}(t) = \sum_{r=0}^n  \left(\sum_{c \ge 1} \, \sum_{r+\nu \le n} \binom{n}{r+\nu}(-1)^c \, \Gamma_{r,c,\nu}\right) \, t^{n-r}
\]where $\Gamma_{r,c,\nu}$ is determined by $\Gamma(x,y,z) = \sum_{r,c,\nu \ge 0} \frac{\Gamma_{r,c,\nu}}{(r+\nu)!} x^ry^cz^\nu$.
\end{theorem}

Principal ideas and results on generating functions are verified in Section~\ref{S:gen-functions}.  They lead to the proof of the main theorem in Section~\ref{S:proofs}.

\

Note that Theorem~\ref{T:char-poly} is readily applicable.  We demonstrate the use of the formula for small values of $n$. The table of characteristic polynomials for $n=3$ up to $n=10$ is provided.  To automate the process,  {\it Mathematica} software is employed.

\section{Decomposition of central graphs}\label{S:decompo}
We recall definitions and key notions from \cite{SONG2} and set up the necessary notations.   Given a $3$-colored central graph $G$, we have a unique decomposition
\[
G = \coprod_{i=0}^3G^{(i)}
\] into four different types of subgraphs: a $3$-colored graph  is  {\it of type (i)}  if
\begin{enumerate}
\item[(i=0)] non-colored, bipartite, without isolated vertices;
\item[(i=1)] non-colored, non-bipartite, without isolated vertices;
\item[(i=2)] totally disconnected and every vertex is colored;
\item[(i=3)] every vertex has a path to a colored vertex.
\end{enumerate}
This decomposition plays an important role in the proof of the main theorem of \cite{SONG1}, and it will again be crucial in proving the main theorem of this article.

\begin{definition}
\begin{enumerate}
\item $\gamma^{(0)}_{r,c,\nu} = $ number of type $(0)$ graphs of rank $r$, cardinality $c$, and with $\nu$ connected components and no isolated vertices;
\item $\gamma^{(i)}_{r,c} =$ number of type $(i)$ graphs of rank $r$, cardinality $c$, and without isolated non-colored vertices, $i = 0, \dots, 3$. Note that $\gamma^{(0)}_{r,c} = \sum_\nu \gamma^{(0)}_{r,c,\nu}$;
\item $ \cgirc =$ number of connected type $(i)$ graphs of rank $r$ and cardinality $c$, $i = 0, \dots, 3$;
\item $\Gamma^{(0)} = \sum_{r,c,\nu \ge 0} \frac{\gamma^{(0)}_{r,c,\nu}}{(r+\nu)!} x^ry^cz^\nu$.
\item $\Gamma^{(i)} = \sum_{r,c \ge 0} \frac{\girc}{r!}x^ry^c$, $i = 1,2,3$;
\end{enumerate}
\end{definition}


\begin{remark}
\begin{enumerate}
\item Note that the type zero subgraphs deserve special treatment because they are the only ones that may fail to have full rank.
\item Since type $(2)$ graphs are totally disconnected, $\gamma^{(2)}_{r,c} = \begin{cases} 0 & r \ne c \\ 2^r & r = c \end{cases}$.
\end{enumerate}
\end{remark}

\begin{prop} $\Gamma := \prod_{i=0}^3 \Gamma^{(i)}$ is the generating function for the number of central graphs of given rank and cardinality. That is, if
\[
\Gamma(x,y,z) = \sum_{n,k,\nu \ge 0} \frac{\Gamma_{n,k,\nu}}{(n+\nu)!} x^n y^k z^\nu
\]
then  $\Gamma_{n,k,\nu}$ is the number of central graphs of rank $n$ and cardinality $k$ with $\nu$ bipartite components. In particular, $\sum_\nu \Gamma_{n,k,\nu}$  is the number of central graphs of rank $n$ and cardinality $k$.
\end{prop}

\begin{proof} Let $n$, $k$, $\nu$ be given. It suffices to compute the number of central graphs $G$ of rank $n$, cardinality $k$ with $\nu$ bipartite components whose type $(i)$ subgraph has rank $n_i$ and cardinality $k_i$.
For $i \ne 0$, $G^{(i)}$ is of full rank and is of order $n_i$. The type $(0)$ subgraph is bipartite and its order is the sum of its rank and the number of components \cite[Theorem~2]{SONG2}. Hence such a graph  would be of order $n + \nu$.

To count the number of graphs $G$ satisfying the conditions above, we first partition $[n+\nu]$ into four sets $V_i$ of vertices such that $|V_0| = n_0+\nu$ and $|V_i| = n_i$, $i = 1,2,3$.  There are $\binom{n+\nu}{n_0+\nu, n_1, n_2, n_3}$ many such partitions. On each fixed $V_i$, there exist $\gamma^{(i)}_{n_i,c_i}$ possible graphs of type $(i)$, $i \ne 0$, and $\gamma^{(0)}_{n_0,c_0,\nu}$ many bipartite graphs on $V_0$ with $\nu$ connected components.

All in all, the number of central graphs of rank $n$ and cardinality $k$ is
\[
\sum_{\nu, n_i, k_i} \binom{n+\nu}{(n_0+\nu)  \, n_1 \, n_2 \,  n_3} \gamma^{(0)}_{n_0,k_0,\nu}\prod_{i=1}^3 \gamma^{(i)}_{n_i,k_i}
\]
where the sum runs over $\nu \in \mathbb Z$ and all partitions $n = \sum_{i=0}^3 n_i$ and $k = \sum_{i=0}^3 k_i$.

General terms of $\Gamma^{(i)}$ are of the form $x^{n_0}y^{k_0}z^\nu$ for $i = 0$ and $x^{n_i}y^{k_i}$ for $i \ne 0$. Hence $x^ny^kz^\nu$ coefficient $\frac{\Gamma_{n,k,\nu}}{(n+\nu)!}$ of $\Gamma$ equals the sum of $\frac{\gamma^{(0)}_{n_0,k_0,\nu}}{(n_0+\nu)!} \prod_{i=1}^3\frac{ \gamma^{(i)}_{n_i,k_i}}{n_i!}$ such that $n = \sum n_i$ and $k = \sum k_i$. Since $\binom{n+\nu}{(n_0+\nu) \, n_1 \, n_2 \, n_3} = \frac{(n+\nu)!}{(n_0+\nu)! n_1! n_2! n_3!}$, the assertion follows.

\end{proof}

\section{Generating functions}\label{S:gen-functions}

\subsection{Generating function for $\gamma^{(0)}_{r,c,\nu}$}\label{S:gamma0}
In \cite{SONG3}, we have computed the generating function for the number of bipartite graphs of given order, size and number of connected components. Here, we modify it slightly to remove the contribution from isolated vertices and translate the information of order and number of connected components to the rank of the graph.

 Let's recall from \cite{SONG3} the generating function for the number of connected bipartite graphs of given order, size and number of connected components. The number $\bar b_{n,k}$ of connected bipartite graphs of  order $n$ and size $k$ is generated by half of the formal logarithm of
 \begin{equation}\label{E:conn-bipartite}
 1 + \sum_{n\ge 1, k\ge 0} \sum_{i=0}^n \binom ni \binom{i(n-i)}k \frac 1{n!}x^ny^k
 \end{equation}
 i.e. its $x^ay^b$ coefficient multiplied by $a!$ gives the number of connected bipartite graphs of order $a$ and size $b$.
Let
\[
\mathcal F(x,y,z) = \exp\left(\sum_{n \ge 2; k \ge 1} \frac{\bar b_{n,k}}{n!} x^n y^k z\right).
\]
 Then
the $x^n y^kz^\nu$-coefficient of $\mathcal F$ multiplied by $n!$ is the number of order $n$, size $k$ bipartite graphs with $\nu$ connected components (\cite[Example~5.2.2]{StanleyBook2}. See also \cite{SONG3}).
Note that $n \ge 2$, $k \ge 1$ since type (0) central graphs do not have isolated vertices.

Since the rank of bipartite graph equals the order minus the number of components, we conclude that the $x^r y^k z^\nu$ coefficient of $\mathcal F(x, y, z/x)$ multiplied by $(r+\nu)!$ equals the number of type (0) central graphs of rank $r$, size $k$ with $\nu$ connected components. We put this neatly into a Proposition:

\begin{prop} (Generating function for type (0) graphs) $\Gamma^0(x,y,z)$ is given by
\[
\exp\left[ \left( \frac12 \log\left(1 + \sum_{n\ge 1, k\ge 0} \sum_{i=0}^n \binom ni \binom{i(n-i)}k \frac 1{n!}x^ny^k\right) - x\right) \frac zx \right].
\]
\end{prop}


\subsection{Generating function for $\gamma^{(1)}_{r,c}$}
By definition, $\gamma^{(1)}_{r,c}$ is obtained by subtracting $\gamma^{(0)}_{r,c}$ from the number $\gamma'_{r,c}$ of all non-colored graphs of rank $r$ and size $c$ without isolated vertices.
\begin{lemma}\label{L:no-isolated-vertex} The number of non-colored graphs of given order and size, and without isolated vertices is generated by
\[
\exp\left(\log \left(1+ \sum_{n \ge 1, k \ge 0} \binom{\binom n2}k \frac1{n!} x^n y^k\right) - x \right)
\]
\end{lemma}
See, for instance, \cite{SONG3} for a proof of the lemma above.


\

The number $\bar b_{n,k}$ of bipartite graphs of order $n$ and size $k$ is computed by Equation~(\ref{E:conn-bipartite}).
Hence  we have:
\begin{prop} The number of type (1) graphs of given order and size is generated by
\[
\Gamma^{(1)}(x,y) := \exp\left(\log \left(1+ \sum_{n \ge 1, k \ge 0} \binom{\binom n2}k \frac1{n!} x^n y^k\right) - x \right) - \sum_{n\ge 2, k\ge 1} \bar b_{n,k} \frac1{n!}x^ny^k
\]
\end{prop}
Since type (1) graphs are of full rank, $\Gamma^{(1)}(x,y)$ precisely generates the number of type (1) graphs of given rank and size!

\subsection{Generating function for $\gamma^{(2)}_{r,c}$}

\begin{prop} The number of type (2) graphs of given order and size is generated by
\[
\Gamma^{(2)}(x,y) = \sum_{r=0}^\infty \frac{2^r}{r!} x^r y^r
\]
\end{prop}

\subsection{Generating function for $\gamma^{(3)}_{r,c}$}
Since connected type (0) or equivalently, bipartite graphs are of rank one less than the order, \cite[Proposition~3]{SONG2} may be re-written as
\[
\overline\gamma^{(3)}_{r,c} = \sum_{t=1}^{c-r+1} 2 \overline{\gamma}^{(0)}_{r-1,c-t}\binom rt.
\]
\begin{lemma} $\Gamma^{(3)}(x,y)$ is given  by  $\exp \overline{\Gamma}^{(3)}$ where
\[
\overline{\Gamma}^{(3)}(x,y) = \sum \frac{\overline{\gamma}^{(3)}_{r,c}}{r!} x^ry^c = \sum \left( \sum_{t=1}^{c-r+1} 2 \overline{\gamma}^{(0)}_{r-1,c-t}\binom rt \right) \frac1{r!} x^r y^c.
\]
\end{lemma}

\begin{proof} The proof is a fairly straightforward application of the Exponential Formula \cite[Corollary~5.1.6]{StanleyBook2}. Let $G$ be a type (3) graph, and consider its decomposition $\coprod G_i$ into connected components, which are again of type (3).
Since type (3) graphs are always of full rank, the rank equals the order. We may now apply the Exponential Formula to obtain the assertion.
\end{proof}

Now we gather the results of the section to give the full generating function for the number of central graphs:
\setthm{1}
\begin{theorem}
The generating function for the number of central graphs is given by
\[
\begin{array}{l}
\Gamma(x,y,z)  =  \exp\left[ \left( \frac12 \log\left(1 + \sum_{n\ge 1, k\ge 0} \sum_{i=0}^n \binom ni \binom{i(n-i)}k \frac 1{n!}x^ny^k\right) - x\right) \frac zx \right] \cdot \\
\left( \sum_{r=0}^\infty \frac{2^r}{r!} x^r y^r\right) \cdot
\left( \sum \left( \sum_{t=1}^{c-r} 2 \overline{\gamma}^{(0)}_{r-1,c-t}\binom rt \right) \frac1{r!} x^r y^c\right) \cdot \\
\left(\exp\left(\log \left(1+ \sum_{n \ge 1, k \ge 0} \binom{\binom n2}k \frac1{n!} x^n y^k\right) - x \right) - \sum_{n\ge 2, k\ge 1} \bar b_{n,k} \frac1{n!}x^ny^k\right)
\end{array}
\]
\end{theorem}

\section{The characteristic polynomial of $\mathcal J_n$} \label{S:proofs}

In this section, we formulate the characteristic polynomial  $\chi_{\mathcal J_n}$.    By the implication of the main theorem of \cite{SONG1}, the $t^{n-r}$-coefficient of $\chi_{\mathcal J_n}$ equals the sum
\[
\sum_{c} (-1)^c \gamma_{r,c}
\]where $\gamma_{r,c}$ is the number of central $3$-colored graphs on $[n]$ of rank $r$ and cardinality $c$.
We have  \setthm{2}
\begin{theorem}
The characteristic polynomial of $\mathcal J_n$ is given by
\[
\chi_{\mathcal J_n}(t) = \sum_{r=0}^n  \left(\sum_{c \ge 1} \, \sum_{r+\nu \le n} \binom{n}{r+\nu}(-1)^c \, \Gamma_{r,c,\nu}\right) \, t^{n-r}.
\]
\end{theorem}

\begin{proof}
 We note that  the coefficient $\Gamma_{r,c,\nu}$ is the number of central graphs on $[r+\nu]$, as opposed to $[n]$, of rank $r$, cardinality $c$ with $\nu$ connected components. So we have
\[
\gamma_{r,c} = \sum_{r+\nu \le n} \binom n{r+\nu} \Gamma_{r,c,\nu}.
\]
\end{proof}

The characteristic polynomials of $\mathcal J_2$ and of $\mathcal J_3$ were computed by hand in \cite{SONG1}. Here, we use the generating function to verify the computation.

For $n = 2$, \[ \Gamma = \Gamma_0\Gamma_2\Gamma_3 = \left(\frac{x y z}{2}+1\right)\left(2 x^2 y^2+4 x y+1\right) \left(x^2 y^3+2 x^2 y^2+1\right).\]  Note that $\Gamma_1$ is trivial. The corresponding characteristic polynomial is $t^2 - 5t + 6$, obtained by using Theorem~\ref{T:char-poly}.

For $n=3,$  the generating function for the central graphs corresponding to $\mathcal J_3$ is
\[
\begin{array}{llll}
\Gamma  = \Gamma_0\Gamma_1\Gamma_2\Gamma_3 & = & \left(\frac{1}{48} x^3 y^3 z^3+\frac{1}{4} x^3 y^3 z^2+\frac{1}{8} x^2 y^2 z^2+\frac{1}{2} x^2 y^2 z+\frac{x y z}{2}+1\right)\cdot\\
&&\left(\frac{x^3 y^3}{6}+1\right) \left(\frac{4 x^3 y^3}{3}+2 x^2 y^2+2 x y+1\right) \cdot \\
&& \left(x^3 y^5+3 x^3 y^4+3 x^3 y^3+x^2 y^3+2 x^2
   y^2+1\right).
   \end{array}
   \]
Using Theorem~\ref{T:char-poly}, we find that
\[ \chi_{\cJ_3}(t) = t^3 - 9t^2 + 27t - 27
\]
which agrees with the computation in \cite{SONG1}.




\appendix
\section*{Appendix}
\addcontentsline{toc}{section}{Appendices}
\renewcommand{\thesubsection}{\Alph{subsection}}

\subsection*{Numerical results}

With the aid of Mathematica, we computed, with increasing running time, characteristic polynomials of higher order.   We list below  characteristic polynomials of degrees up to $n=10.$

\[\begin{aligned}
\mathcal  \chi_{\mathcal J_2} (t)  & = t^2 - 5t + 6 \\
\mathcal \chi_{\mathcal J_3 }(t)  & = t^3 - 9t^2 + 27t - 27\\
\mathcal \chi_{\mathcal J_4 }(t)  & = t^4-14t^3+75t^2-168t+104     \\
\mathcal \chi_{\mathcal J_5 }(t)  & = t^5 -20 t^4+165 t^3-695 t^2+1465
   t-3649    \\
\mathcal \chi_{\mathcal J_6 }(t)  & = t^6 -27 t^5+315 t^4-2010 t^3+7365 t^2-9285 t+97605   \\
\mathcal \chi_{\mathcal J_7 }(t)  & = t^7-35 t^6+546 t^5-4865 t^4+26565 t^3-92386 t^2-252245 t-3082889 \\
\mathcal \chi_{\mathcal J_8 }(t)  & = t^8 -44 t^7+882 t^6-10402 t^5+78365 t^4-382662 t^3  +1959447 t^2+22977452 t \\  & \quad \quad +104683724 \\
\mathcal \chi_{\mathcal J_9 }(t)  & = t^9 -54 t^8+1350 t^7-20286 t^6+200403 t^5-1338708 t^4+8421021 t^3  \\ \quad & \quad \quad+105101892 t^2 +1112954274 t+866974176 \\
\mathcal \chi_{\mathcal J_{10} }(t)  & = t^{10}  -65 t^9+1980 t^8-36840 t^7+460215 t^6-4008081 t^5+24881535
   t^4 \\ & \quad\quad+52962615 t^3+7605232140 t^2+71654230070 t+142378721936
\end{aligned}
\]

According to Theorem~\ref{T:Zas75}, the number of bounded chambers  in $\mathbb R^n$ divided by hyerplanes in $\mathcal{J}_n$ are

\[ \begin{array}{|c|c||c|c|} \hline
n & \quad  (-1)^n \mathcal \chi_{\mathcal J_{n} } (-1) \quad  & n &   \quad (-1)^n\mathcal \chi_{\mathcal J_{n} } (-1) \quad  \\ \hline
 3 &  64 &    7 &   170 770  \\
  4 &  362 &      8 & 84138075  \\
 5 &   5995 &   9 & 150860029   \\
 6 &   116608 &    10 &  78306150108   \\
  \hline
\end{array}
\]



\bibliographystyle{amsplain}

\bibliography{generate}


\end{document}